\documentclass{amsart}
\numberwithin{equation}{section}

\newtheorem{thm}{Theorem}[section]
\newtheorem{lem}[thm]{Lemma}

\newtheorem{defin}[thm]{Definition}
\newtheorem{rem}[thm]{Remark}

\begin{document}
\title{Initial-boundary value problem}

\author{Ravshan Ashurov}
\author{Oqila Muhiddinova}
\address{National University of Uzbekistan named after Mirzo Ulugbek and Institute of Mathematics, Uzbekistan Academy of Science}
\curraddr{Institute of Mathematics, Uzbekistan Academy of Science,
Tashkent, 81 Mirzo Ulugbek str. 100170} \email{ashurovr@gmail.com}

\small

\title[Initial-boundary value problem]
{Initial-boundary value problem for a time-fractional subdiffusion
equation on the torus}

\begin{abstract}

An initial-boundary value problem for a time-fractional
subdiffusion equation with the Riemann-Liouville derivatives on
$N$-dimensional torus is considered. Uniqueness and existence of
the classical solution of the posed problem are proved by the
classical Fourier method. Sufficient conditions for the initial
function and for the right-hand side of the equation are
indicated, under which the corresponding Fourier series converge
absolutely and uniformly. It should be noted, that the condition
on the initial function found in this paper is less restrictive
than the analogous condition in the case of an equation with
derivatives in the sense of Caputo.

\end{abstract}+-

\maketitle

\section{Main result}

The fractional integration of order $ \sigma <0 $ of the function
$ h (t) $ defined on $ [0, \infty) $ has the form
$$
\partial_t^\sigma h(t)=\frac{1}{\Gamma
(-\rho)}\int\limits_0^t\frac{h(\xi)}{(t-\xi)^{\sigma+1}} d\xi,
\quad t>0,
$$
provided the right-hand side exists. Here $\Gamma(\sigma)$ is
Euler's gamma function. Using this definition one can define the
Riemann - Liouville fractional derivative of order $\rho$,
$0<\rho< 1$, as (see, for example, \cite{PSK}, p. 14)
$$
\partial_t^\rho h(t)= \frac{d}{dt}\partial_t^{\rho-1} h(t).
$$
If in this definition we interchange differentiation and
fractional integration, then we get the definition of the
regularized derivative, that is, the definition of the fractional
derivative in the sense of Caputo:
$$
D_t^\rho h(t)= \partial_t^{\rho-1}\frac{d}{dt} h(t).
$$

Note that if $\rho=1$, then fractional derivatives coincides with
the ordinary classical derivative of the first order: $\partial_t
h(t) = D_t h(t)= \frac{d}{dt} h(t)$.

Let $\mathbb{T}^N$ be $N$ -dimensional torus: $\mathbb{T}^N=(-\pi,
\pi]^N$, $N\geq 1$. We define by $C(\mathbb{T}^N)$ and
$C^2(\mathbb{T}^N)$ a class of $2\pi$-periodic on each variable
$x_j$ functions $v(x)$ from $C(\mathbb{T}^N)$ and
$C^2(\mathbb{T}^N)$ correspondingly. Let $A$ stand for a positive
operator, defined on $C^2(\mathbb{T}^N)$ and acting as $A
v(x)=-\Delta v(x)$, where $\Delta$ is the Laplace operator.

Let $\rho\in(0,1) $ be a constant number. Consider the
initial-boundary value problem
\begin{equation}\label{eq}
\partial_t^\rho u(x,t) + Au(x,t) = f(x,t), \quad x\in \mathbb{T}^N, \quad
0<t\leq T,
\end{equation}
\begin{equation}\label{in}
\lim\limits_{t\rightarrow 0}\partial_t^{\rho-1} u(x,t) =
\varphi(x), \quad x\in \mathbb{T}^N,
\end{equation}
where $f$ and $\varphi$ are given continuous functions.

\begin{defin} A function $u(x,t)$ with the properties
$\partial_t^\rho u(x,t), A(x,D)u(x,t)\in C(\mathbb{T}^N\times
(0,T])$, $\partial_t^{\rho-1} u(x,t)\in C(\mathbb{T}^N\times
[0,T])$  and satisfying the conditions of problem (\ref{eq}) -
(\ref{in}) is called \textbf{the solution} of the initial-boundary
value problem.
\end{defin}

Before formulating the main result, let us introduce some
concepts.

It is not hard to see that the closure $\hat{A}$ of operator $A$
in $L_2(\mathbb{T}^N)$ is selfadjoint and it has a complete (in
$L_2(\mathbb{T}^N)$) set of eigenfunctions $\{\gamma e^{inx}\}$,
$\gamma = \gamma(N)=(2\pi)^{-N/2}$, $n\in \mathbb{Z}^N$ and
corresponding eigenvalues $|n|^2=n_1^2+n_2^2+...+n_N^2$.
Therefore, by virtue of J. von Niemann theorem, for any $\tau> 0$
one can introduce the power of operator $\hat{A}$ as $\hat{A}^\tau
g(x)=\sum\limits_{n\in \mathbb{Z}^N} |n|^\tau g_n e^{inx}$, where
$g_n$ are Fourier coefficients:
$$
g_n=(2\pi)^{-N}\int\limits_{\mathbb{T}^N} g(x) e^{-inx} dx.
$$
The domain of definition of this operator is defined from the
condition $\hat{A}^\tau g(x)\in L_2(\mathbb{T}^N)$ and has the
form
$$
D(\hat{A}^\tau)=\{g\in L_2(\mathbb{T}^N): \sum\limits_{n\in
\mathbb{Z}^N} |n|^{2\tau} |g_n|^2 < \infty\}.
$$
On the other hand, the class of functions $L_2(\mathbb{T}^N)$
which for a given fixed number $a> 0$ make the norm
\begin{equation}\label{T}
||g||^2_{L_2^a(\mathbb{T}^N)}=\big|\big|\sum\limits_{n\in
\mathbb{Z}^N}(1+|n|^2)^{\frac{a}{2}}g_n
e^{inx}\big|\big|^2_{L_2(\mathbb{T}^N)}=\sum\limits_{n\in
\mathbb{Z}^N}(1+|n|^2)^a|g_n|^2
\end{equation}
finite is termed the Liouville class $L_2^a(\mathbb{T}^N)$.
Therefore one has $ D(\hat{A}^\tau)= L_2^{\tau m}(\mathbb{T}^N)$.

Let $E_{\rho, \mu}$ be  the two-parametric Mittag-Leffler
function:
$$
E_{\rho, \mu}(t)= \sum\limits_{k=0}^\infty \frac{t^k}{\Gamma(\rho
k+\mu)}.
$$

Here is the main result.

\begin{thm}\label{main} Let $a > \frac{N}{2}$ and $\varphi\in C(\mathbb{T}^N)\cap L^{a-2}_2(\mathbb{T}^N)$. Moreover, let $f(x,t)\in
L^a_2(\mathbb{T}^N)$ for $0<t\leq T$
 and $||t^{1-\rho}\,f(\cdot, t)||^2_{L_2^a(\mathbb{T}^N)}\in C[0, T]$. Then there exists a solution of initial-boundary
value problem (\ref{eq}) - (\ref{in}) and it has the form
\begin{equation}\label{solution}
u(x,t)=\sum\limits_{n\in\mathbb{Z}^N} \bigg[\varphi_n t^{\rho-1}
E_{\rho, \rho} (-|n|^2 t^\rho)+\int\limits_0^t
f_n(t-\xi)\xi^{\rho-1} E_{\rho, \rho}(-|n|^2\xi^\rho)
d\xi\bigg]e^{inx},
\end{equation}
which absolutely and uniformly converges on $x\in \mathbb{T}^N$
and for each $t\in (0, T]$, where $\varphi_n$ and $f_n(t)$ are
corresponding Fourier coefficients. Moreover, the series obtained
after applying term-wise the operators $\partial_t^\rho$ and $A$
also converge absolutely and uniformly on $x\in \mathbb{T}^N$ and
for each $t\in (0, T]$.

\end{thm}
\begin{rem}
If $N \leq 3$, then under the conditions of the theorem it
suffices to require $\varphi\in C(\mathbb{T}^N)$. Note, when
$a>\frac{N}{2}$, according to the Sobolev embedding theorem, all
functions in $L^{a}_2(\mathbb{T}^N)$ are $2\pi$-periodic
continuous functions. The fulfillment of the inverse inequality $a
\leq \frac{N}{2}$, admits the existence of unbounded functions in
$L^{a}_2(\mathbb{T}^N)$ (see, for example, \cite{AAP}). Therefore,
condition  $a>\frac{N}{2}$ for function $f$ of this theorem is not
only sufficient for the statement to be hold, but it is also
necessary.
\end{rem}

Initial-boundary value problem (\ref{eq}) - (\ref{in}) for various
elliptic operators $A$ have been considered by a number of authors
using different methods (see, for example, handbook \cite{Koch}).
It has been mainly considered the case of the Caputo derivatives
$D_t^\rho$ instead of $\partial_t^\rho$. In the book of A.A.
Kilbas et al. \cite{Kil} (Chapter 6) there is a survey of works
published before 2006. The case of one spatial variable $x\in
\mathbb{R}$ and subdiffusion equation with "the elliptical part"
$u_{xx}$ were considered  for example in the monograph of A. V.
Pskhu \cite{PSK} (Chapter 4, see references thesein). The paper
Gorenflo, Luchko and Yamamoto \cite{GorLuchYam} is devoted to the
study of subdiffusion equations in Sobelev spaces. In the paper by
Kubica and Yamamoto \cite{KubYam}, initial-boundary value problems
for equations with time-dependent coefficients are considered. In
the multidimensional case ($x\in \mathbb{R}^N$), instead of the
differential expression $u_{xx}$, authors considered either the
second order elliptic operator (\cite{Agr} - \cite{PS1}) or
elliptic pseudodifferential operators with constant coefficients
in the whole space $\mathbb{R}^N$ (Umarov \cite{SU}). In the paper
of Yu. Luchko \cite{Luch} the author constructed solutions by the
eigenfunction expansion in the case of $f = 0$ and discussed the
unique existence of the generalized solution to problem (\ref{eq})
- (\ref{in}) with the Caputo derivative.  In his recent paper
\cite{PS1} A. V. Pskhu considered an initial-boundary value
problem for subdiffusion equation with the Laplace operator and
domain $\Omega$ - a multidimensional rectangular region. The
author succeeded to construct the Green's function. In an
arbitrary $N$-dimensional domain $\Omega$ initial-boundary value
problems for subdiffusion equations (the fractional part of the
equation is a multi-term and initial conditions are non-local)
with the Caputo derivatives has been investigated by M. Ruzhansky
et al. \cite{Ruz}. The authors proved the existence and uniqueness
of the generalized solution to the problem.

A result similar to Theorem \ref{main} was obtained in the recent
paper \cite{AO} for a more general subdiffusion equation. But the
conditions on the functions $f(x, t)$ and $\varphi(x)$ that
guarantee the existence and uniqueness of the solution to problem
(\ref{eq})-(\ref{in}) found in \cite{AO} are more stringent. This
is due to the fact that in the present paper we give a more
precise estimate for the Mittag-Leffler function $E_{\rho,
\rho}(-t)$, $t>0$ (see also \cite{AA}).

It is interesting to note that the condition on the function
$\varphi(x)$ found in Theorem \ref{main} is less restrictive than
the analogous condition in the case of an equation (\ref{eq}) with
derivatives in the sense of Caputo (see, for example, \cite{Luch},
\cite{Yama11}).

\section{Proof of Theorem \ref{main}}

The uniqueness of the solution can be proved by the standard
technique based on completeness in $L_2(\mathbb{T}^N)$ of the set
of eigenfunctions $\{\gamma e^{inx}\}$ (see, for example,
\cite{AO}).

Proof of existence based on the following lemma (see M.A.
Krasnoselski et al. \cite{Kra}, p. 453), which is a simple
corollary of the Sobolev embedding theorem.

\begin{lem}\label{CL} Let $\sigma > 1+\frac{N}{4}$. Then for any $|\alpha|\leq 2$
operator $D^\alpha (\hat{A}+1)^{-\sigma}$ (completely)
continuously maps from $L_2(\mathbb{T}^N)$ into $C(\mathbb{T}^N)$
and moreover the following estimate holds true
\begin{equation}\label{CL1}
||D^\alpha (\hat{A}+1)^{-\sigma} g||_{C(\mathbb{T}^N)} \leq C
||g||_{L_2(\mathbb{T}^N)}.
\end{equation}

\end{lem}
\begin{proof}Since the embedding theorem $||D^\alpha (\hat{A}+1)^{-\sigma} g||_{C(\mathbb{T}^N)} \leq C||D^\alpha (\hat{A}+1)^{-\sigma}
g||_{L_2^a(\mathbb{T}^N)}$ for $a>N/2$, then it is sufficient to
prove the inequality $||D^\alpha (\hat{A}+1)^{-\sigma}
g||_{L_2^a(\mathbb{T}^N)} \leq C||g||_{L_2(\mathbb{T}^N)}.$ But
this is a consequence of the estimate
$$
\sum\limits_{n\in\mathbb{Z}^N}|g_n|^2|n|^{2|\alpha|}(1+A(n))^{-2\sigma}(1+|n|^2)^a\leq
C\sum\limits_{n\in\mathbb{Z}^N}|g_n|^2,
$$
that is valid for $\frac{N}{2}< a \leq 2\sigma  - |\alpha|$.

\end{proof}

Since the initial condition (\ref{in}) can be rewritten as (see,
for example, \cite{PSK} p. 104)
\begin{equation}\label{in1}
\lim\limits_{t\rightarrow
0}t^{1-\rho}u(x,t)=\frac{\varphi(x)}{\Gamma(\rho)},
\end{equation}
then one can easily verify that the function (\ref{solution})
formally satisfies the conditions of problem (\ref{eq})-(\ref{in})
(see, for example, \cite{Gor}, p. 173). In order to prove that
function (\ref{solution}) is actually a solution to the problem,
it remains to substantiate this formal statement, i.e. show that
the operators $A$ and $\partial_t^\rho$ can be applied term by
term to the series (\ref{solution}). To do this we remind the
following asymptotic estimate of the Mittag-Leffler function with
a sufficiently large negative argument (see, for example,
\cite{Dzh66}, p. 134)
\[
E_{\rho, \rho}(-t)=-\frac{t^{-2}}{\Gamma(-\rho)} +O(t^{-3}).
\]
Therefore, since $E_{\rho, \rho}(t)$ is real analytic, one has the
estimate
\begin{equation}\label{m1}
|E_{\rho, \rho}(-t)|\leq \frac{C}{1+ t^2}, \quad t>0.
\end{equation}
We will also use a coarser estimate with a positive $\lambda$ and
$0<\varepsilon<1$:
\begin{equation}\label{m2}
|t^{\rho-1} E_{\rho,\rho}(-\lambda t^\rho)|\leq
\frac{Ct^{\rho-1}}{1+(\lambda t^\rho)^2}\leq C
\lambda^{\varepsilon-1} t^{\varepsilon\rho-1}, \quad t>0,
\end{equation}
which is easy to verify. Indeed, let $t^\rho\lambda<1$, then $t<
\lambda^{-1/\rho}$ and
$$
t^{\rho -1} = t^{\rho-\varepsilon\rho} t^{\varepsilon\rho-1} <
\lambda^{\varepsilon-1}t^{\varepsilon\rho-1}.
$$
If $t^\rho\lambda\geq 1$, then $\lambda^{-1}\leq t^\rho$ and
$$
\lambda^{-2} t^{-\rho-1}=\lambda^{-1+\varepsilon}
\lambda^{-1-\varepsilon} t^{-\rho-1}\leq
\lambda^{\varepsilon-1}t^{\varepsilon\rho-1}.
$$

Note the series (\ref{solution}) is in fact the sum of two series.
Consider the following partial sums of the first series:
\begin{equation}\label{S1}
S^1_k(x,t)=\sum\limits_{|n|^2<k} t^{\rho-1} E_{\rho,\rho}(-|n|^2
t^\rho)\varphi_n \, e^{inx},
\end{equation}
and suppose that function $\varphi$ satisfies the condition of
Theorem \ref{main}, i.e. for some $\tau> \frac{N}{4}$
$$
\sum\limits_{n\in\mathbb{Z}^N} |n|^{4(\tau-1)} |\varphi_n|^2 \leq
C_\varphi<\infty.
$$
Since $\hat{A}^{-\tau-1} e^{inx} = |n|^{-2(\tau+1)} e^{inx}$, we
may rewrite the sum (\ref{S1}) as
$$
S^1_k(x,t)=\hat{A}^{-\tau-1}\sum\limits_{|n|^2<k} t^{\rho-1}
E_{\rho,\rho}(-|n|^2 t^\rho)\varphi_n \, |n|^{2(\tau+1)}\,
e^{inx}.
$$
Therefore by virtue of Lemma \ref{CL} one has
$$
||D^\alpha S^1_k||_{C(\mathbb{T}^N)}=||D^\alpha
\hat{A}^{-\tau-1}\sum\limits_{|n|^2<k} t^{\rho-1}
E_{\rho,\rho}(-|n|^2 t^\rho)\varphi_n \, |n|^{2(\tau+1)}\,
e^{inx}||_{C(\mathbb{T}^N)}\leq
$$
\begin{equation}\label{S11}
\leq C ||\sum\limits_{|n|^2<k} t^{\rho-1} E_{\rho,\rho}(-|n|^2
t^\rho)\varphi_n \, |n|^{2(\tau+1)}\,
e^{inx}||_{L_2(\mathbb{T}^N)}.
\end{equation}
Using the orthonormality of the system $\{e^{inx}\}$, we will have
\begin{equation}\label{S2}
||D^\alpha S^1_k||^2_{C(\mathbb{T}^N)}\leq C \sum\limits_{|n|^2<k}
\big|t^{\rho-1} E_{\rho,\rho}(-|n|^2 t^\rho)\varphi_n \,
|n|^{2(\tau+1)}\big|^2.
\end{equation}
Application of estimate (\ref{m1}) and inequality $(|n|^2
t^\rho)^2 (1+ (|n|^2 t^\rho)^2)^{-1}<1$ gives
$$
\sum\limits_{|n|^2<k} \big|t^{\rho-1} E_{\rho,\rho}(-|n|^2
t^\rho)\varphi_n \, |n|^{2(\tau+1)}\big|^2\leq C
t^{-2(\rho+1)}\sum\limits_{|n|^2<k}
|n|^{4(\tau-1)}|\varphi_n|^2\leq C t^{-2(\rho+1)}C_\varphi.
$$
Therefore we can rewrite the estimate (\ref{S2}) as
$$
||D^\alpha S^1_k||^2_{C(\mathbb{T}^N)}\leq C t^{-2(\rho+1)}
C_\varphi.
$$

This implies uniformly on $x\in\mathbb{T}^N$ convergence of the
differentiated sum (\ref{S1}) with respect to the variables $x_j$
for each $t\in (0,T]$. On the other hand, the sum (\ref{S11})
converges for any permutation of its members as well, since these
terms are mutually orthogonal. This implies the absolute
convergence of the differentiated sum (\ref{S1}) on the same
interval $t\in (0,T]$.

Now we consider the second part of the series (\ref{solution}):
\begin{equation}\label{S2}
S^2_k(x,t)=\sum\limits_{|n|^2<k} \int\limits_0^t
f_n(t-\xi)\xi^{\rho-1} E_{\rho,\rho}(-|n|^2 \xi^\rho)\,d\xi\,
e^{inx}
\end{equation}
and suppose that function  $f(x,t)$  satisfies all the conditions
of Theorem \ref{main}, i.e. the following series converges
uniformly on $t\in [0, T]$ for some $\tau> \frac{N}{4}$:
$$
\sum\limits_{n\in\mathbb{Z}^N}t^{2(1-\rho)} |n|^{4\tau} |f_n(t)|^2
\leq C_f<\infty.
$$
We choose a small $\varepsilon>0$ in such a way, that
$\tau+1-\varepsilon> 1+\frac{N}{4}$.  Since
$\hat{A}^{-\tau-1+\varepsilon} e^{inx} =
|n|^{-2(\tau+1-\varepsilon)} e^{inx}$, we may rewrite the sum
(\ref{S2}) as
$$
S^2_k(x,t)=\hat{A}^{-\tau-1+\varepsilon}\sum\limits_{|n|^2<k}
\int\limits_0^t f_n(t-\xi)\xi^{\rho-1} E_{\rho,\rho}(-|n|^2
\xi^\rho)\,d\xi\,|n|^{2(\tau+1-\varepsilon)} e^{inx}.
$$
Then by virtue of Lemma \ref{CL} one has
$$
||D^\alpha S^2_k||_{C(\mathbb{T}^N)}=||D^\alpha
\hat{A}^{-\tau-1+\varepsilon}\sum\limits_{|n|^2<k} \int\limits_0^t
f_n(t-\xi)\xi^{\rho-1} E_{\rho,\rho}(-|n|^2
\xi^\rho)\,d\xi\,|n|^{2(\tau+1-\varepsilon)}
e^{inx}||_{C(\mathbb{T}^N)}\leq
$$
\[
\leq C \big|\big|\sum\limits_{|n|^2<k} \int\limits_0^t
f_n(t-\xi)\xi^{\rho-1} E_{\rho,\rho}(-|n|^2
\xi^\rho)\,d\xi\,|n|^{2(\tau+1-\varepsilon)}
e^{inx}\big|\big|_{L_2(\mathbb{T}^N)}.
\]
Using the orthonormality of the system $\{e^{inx}\}$, we will have
$$
||D^\alpha S^2_k||^2_{C(\mathbb{T}^N)}\leq C \sum\limits_{|n|^2<k}
\big|\int\limits_0^t f_n(t-\xi)\xi^{\rho-1} E_{\rho,\rho}(-|n|^2
\xi^\rho)\,d\xi\,|n|^{2(\tau+1-\varepsilon)}\big|^2.
$$
Now we use estimate (\ref{m2}) and apply the generalized Minkowski
inequality. Then
$$
||D^\alpha S^2_k||^2_{C(\mathbb{T}^N)}\leq C\bigg(\int\limits_0^t
\xi^{\varepsilon\rho-1} (t-\xi)^{\rho-1}
\big(\sum\limits_{|n|^2<k}
|n|^{4\tau}(t-\xi)^{2(1-\rho)}|f_n(t-\xi)|^2\big)^{1/2} d\xi
\bigg)^2\leq C\cdot C_f,
$$
where $C$ depends on $T$ and $\varepsilon$. Hence, using the same
argument as above, we see that the differentiated sum (\ref{S2})
with respect to the variables $x_j$ converges absolutely and
uniformly on $(x,t)\in \mathbb{T}^N\times [0,T]$.

Further, from equation (\ref{eq}) one has
$$
\partial_t^\rho \big(S^1_k+S^2_k\big)=  -A \big(S^1_k+S^2_k\big)+\sum\limits_{|n|^2<k} f_n(t)e^{inx}.
$$
Absolutely and uniformly convergence of the latter  series can be
proved as above.

Thus Theorem \ref{main} is completely proved.

\section{Acknowledgement} The authors convey thanks to Sh. A.
Alimov for discussions of these results.

\

\

\bibliographystyle{amsplain}

\end{document}